\documentclass{amsart}
\usepackage{mathrsfs,amssymb,setspace}

\theoremstyle{plain}
\newtheorem{theorem}{Theorem}[section]
\newtheorem{lemma}[theorem]{Lemma}
\newtheorem{proposition}[theorem]{Proposition}
\newtheorem{corollary}[theorem]{Corollary}

\theoremstyle{definition}

\newtheorem{definition}[theorem]{Definition}

\theoremstyle{remark}

\input xy
\xyoption{all}

\def\G{\Gamma}

\begin{document}


\title[The Large Rank of a Finite Semigroup]{The Large Rank of a Finite Semigroup\\ Using Prime Subsets}
\author[Jitender Kumar, K. V. Krishna]{Jitender Kumar and  K. V. Krishna}
\address{Department of Mathematics, Indian Institute of Technology Guwahati, Guwahati, India}
\email{\{jitender, kvk\}@iitg.ac.in}

\subjclass[]{20M10}

\keywords{Large rank, Brandt semigroups, Transformation semigroups}

\maketitle

\begin{center}
Communicated by Mikhail Volkov
\end{center}


\begin{abstract}
The \emph{large rank} of a finite semigroup $\G$, denoted by $r_5(\G)$, is the least number $n$ such that every subset of $\G$ with $n$ elements generates $\G$. Howie and Ribeiro showed that $r_5(\G) = |V| + 1$, where $V$ is a largest proper subsemigroup of $\G$. This work considers the complementary concept of subsemigroups, called \emph{prime subsets}, and gives an alternative approach to find the large rank of a finite semigroup. In this connection, the paper provides a shorter proof of Howie and Ribeiro's result about the large rank of Brandt semigroups. Further, this work obtains the large rank of the semigroup of order-preserving singular selfmaps.
\end{abstract}


\section{Introduction}

The concept of rank for general algebras is analogous to the concept of dimension in linear algebra. The dimension of a vector space is the maximum cardinality of an independent subset, or equivalently, it is the minimum cardinality of a generating set of the vector space.
A subset $U$ of a semigroup $\G$ is said to be \emph{independent} if every element of $U$ is not in the subsemigroup generated by the remaining elements of $U$, i.e., \[ \forall a \in U, \; a \notin \langle U \setminus \{a\} \rangle .\] This definition of independence is analogous to the usual definition of independence in linear algebra. Howie and Ribeiro observed that the minimum size of a generating set need not be equal to the maximum size of an  independent set in a semigroup. Accordingly, they have considered the following possible definitions of ranks for a finite semigroup $\G$ (cf. \cite{a.hw99,a.hw00}).
\begin{enumerate}
\item $r_1(\G) = \max\{k :$ every subset $U$  of cardinality $k$ in $\G$ is independent\}.
\item $r_2(\G) = \min\{|U| : U \subseteq \G, \langle U\rangle = \G\}$.
\item $r_3(\G) = \max\{|U| : U \subseteq \G, \langle U\rangle = \G, U$ is independent\}.
\item $r_4(\G) = \max\{|U| : U \subseteq \G, U$  is independent\}.
\item $r_5(\G) = \min\{k : $ every subset $U$ of cardinality $k$ in  $\G$  generates $\G$\}.
\end{enumerate}
For a finite semigroup $\G$, it can be observed that \[r_1(\G) \le r_2(\G) \le r_3(\G) \le r_4(\G) \le r_5(\G).\] Thus,
$r_1(\G), r_2(\G), r_3(\G), r_4(\G)$ and $r_5(\G)$ are, respectively, known as \emph{small rank}, \emph{lower rank}, \emph{intermediate rank}, \emph{upper rank} and \emph{large rank} of $\G$.

In this work,  we focus on determining the large rank of a finite semigroup. The large rank can be obtained by using the following key result of Howie and Ribeiro.
\begin{theorem}[\cite{a.hw00}]\label{r5-lsgp}
Let $\G$ be a finite semigroup and let $V$ be a  proper subsemigroup of $\G$ with the largest possible size. Then $r_5(\G) = |V| + 1.$
\end{theorem}
Consequently, if a finite semigroup $\G$ has an indecomposable element, then $r_5(\G) = |\G|$. Thus, by identifying indecomposable elements, authors determined large ranks of certain semigroups which include free nilpotent semigroups, free semilattices, finite monogenic semigroups which are not groups (cf. \cite{a.hw00,a.mini09,t.jdm02}). Howie and Ribeiro obtained large ranks of the full transformation semigroup and an arbitrary Brandt semigroup by investigating largest proper subsemigroups of these semigroups \cite{a.hw00}.

In this paper, we propose an alternative approach to ascertain the large rank of a finite semigroup using the notion called prime sets. In the proposed approach, we provide an elegant and shorter proof for the large rank of Brandt semigroups. Further, in this paper, we also obtain the large rank for the semigroup of order-preserving singular selfmaps.

\section{Prime subsets vis-\`a-vis large rank}

In this section, we introduce the concept of prime subsets of a semigroup using which we propose a technique to find the large rank of a finite semigroup.

\begin{definition}
A nonempty subset $U$ of a (multiplicative) semigroup $\G$ is said to be \emph{prime} if, for all $a, b \in \G$, \[ab \in U \; \; \text{implies}\; \; a \in U \mbox{ or } b \in U.\]
\end{definition}

The following proposition is useful in the sequel.

\begin{proposition}\label{gm-lsgp}
Let $V$ be a proper subset of a finite semigroup $\G$. Then $V$ is a smallest prime subset of $\G$ if and only if $\G \setminus V$ is a largest subsemigroup of $\G$.
\end{proposition}

\begin{proof}
Note that $\G \setminus V$ is not a subsemigroup of $\G$  if and only if  there exist $a, b \in \G \setminus V$ such that $ab \not\in \G \setminus V$ if and only if  there exist $a, b \not\in V$ such that  $ab \in V$ if and only if  $V$ is not a prime subset.
Now, for any $X, Y \subset \G$, since $|X| + |\G \setminus X| = |Y| + |\G \setminus Y|$, we have $|X| < |Y|$ if and only if $|\G \setminus Y| < |\G \setminus X|$. Hence, we have the result.
\end{proof}

In view of Theorem \ref{r5-lsgp}, we have the following corollary of Proposition \ref{gm-lsgp}.

\begin{corollary}\label{c.tech}
Let $V$ be a smallest proper prime subset of a finite semigroup $\G$. Then $r_5(\G) = |\G \setminus V| + 1.$
\end{corollary}

Thus, the problem of finding the large rank of a finite semigroup is now reduced to the problem of finding a smallest proper prime subset of the semigroup.

\section{Brandt semigroups}

In this section, we give much shorter proof of Howie and Ribeiro's result about the large rank of Brandt semigroups (cf. \cite{a.hw00}).

Given a finite group $G$ and a natural number $n$, write $[n] = \{1,2, \ldots, n\}$ and $B(G, n) = ([n] \times G \times [n]) \cup \{0\}$.  Define a binary operation (say, multiplication) on $B(G, n)$ by
\[ (i, a, j)(k, b, l) =
                \left\{\begin{array}{cl}
                (i, ab, l) & \text {if $j = k$;}  \\
                0     & \text {otherwise,}
                \end{array}\right. \]
\[\mbox{and }\; 0(i, a, j)  = (i, a, j)0 = 00 = 0.\]
With the above defined multiplication, $B(G, n)$ is a semigroup known as a \emph{Brandt semigroup}. When $G$ is the trivial group, the Brandt semigroup is denoted by $B_n$.

\begin{theorem}\label{lr-bs}
For $n \ge 2$, $r_5(B(G, n)) = (n^2 - n + 1)|G| + 2.$
\end{theorem}

\begin{proof}
We claim that the set $V = \{(n, a, k) \mid a \in G \; \text{and}\;1 \le k \le n - 1\}$ is a smallest prime subset of $B(G, n)$. Since $|V| = (n-1)|G|$, the result follows.

\emph{$V$ is a prime subset}: For $(i, a, j), (k, b, l) \in B(G, n)$, if $(i, a, j)(k, b, l) \in V$, then $i = n$, $j = k$ and $1 \le l \le n-1$. If $k = n$, then clearly $(k, b, l) \in V$; otherwise, $(i, a, j) \in V$.

\emph{$V$ is a smallest prime subset}: Let $U$ be a prime subset of $B(G, n)$ such that $|U| < |V|$. If $U \subset V$, then let $(n, a, q) \in V \setminus U$. Now, for $(n, b, p) \in U$ and for all $i \in [n]$, clearly we have \[(n, b, p) = (n, a, i)(i, a^{-1}b, p).\] Note that, for $i = q$, neither $(n, a, i)$ nor $(i, a^{-1}b, p)$ is in $U$; a contradiction to $U$ is a prime set.

Otherwise, we have $U \not \subset V$. Let $\theta \in U \setminus V$; then, $\theta$ can be (i) $0$, (ii) $(n, a, n)$, or (iii) $(p, a, q)$, for some $p \in [n-1]$, $q \in [n]$ and $a \in G$. In all the three cases we observe that $|U| \ge (n-1)|G|$, which is a contradiction to the choice of $U$.

(i) $\theta = 0$: For each $b \in G$ and $i \in [n]$, since $(i, b, 1)(2, b, i) = 0$, there are at least $n|G|$ elements in $U$.

(ii) $\theta = (n, a, n)$, for some $a \in G$: For each $b \in G$ and $i \in [n-1]$, since $(n, b, i)(i, b^{-1}a, n) = (n, a, n)$, there are at least $(n-1)|G|$ elements in $U$.

(iii) $\theta = (p, a, q)$, for some $p \in [n-1]$, $q \in [n]$ and $a \in G$. First note that, for each $b \in G$ and $i \in [n]$, we have $(p, b, i)(i, b^{-1}a, q) = (p, a, q)$. If $p = q$, the argument is similar to above (ii). Otherwise, corresponding to $n-2$ different choices of $i \ne p, q$, there are at least $(n - 2)|G|$ elements in $U$. In addition, for $i = p$ or $q$, $U$ must contain at least $|G|$ more elements. Thus, $|U| \ge (n-1)|G|$.
\end{proof}

\begin{corollary}
For $n \ge 2$, $r_5(B_n) = n^2-n +3.$
\end{corollary}

\section{Order-preserving singular selfmaps}

In this section, we consider the semigroup \[\mathcal{O}_n = \{\alpha \in {\rm Sing}_n : ( \forall x, y \in [n]) \; x \leq y \Rightarrow x \alpha \leq y \alpha\}\] of order-preserving singular selfmaps with respect to composition of mappings, where ${\rm Sing}_n$ is the set of all singular maps (non-permutations) from $[n]$ to $[n]$. In \cite{a.hw92}, Gomes and Howie considered the semigroup $\mathcal{O}_n$ and found that $r_2(\mathcal{O}_n) = n$, for $n \ge 2$. Note that $r_1(\mathcal{O}_2) = 2$. For $n \ge 3$, since $\mathcal{O}_n$ is not a band, it can be ascertained that $r_1(\mathcal{O}_n) = 1$ (cf.  \cite[Theorem 2]{a.hw00}). However, the ranks $r_3, r_4$ and $r_5$ of the semigroup $\mathcal{O}_n$ are not known. In this work, we obtain the large rank of the semigroup $\mathcal{O}_n$. It can be easily observed that $r_5(\mathcal{O}_2) = 2$. In what follows, $n \ge 3$.

We recall the Green's $\mathcal{J}$-classes in $\mathcal{O}_n$ from \cite{a.hw92}. Since
$\alpha \; \mathcal{J} \; \beta$ if and only if $|{\rm im } \; \alpha| = |{\rm im} \; \beta|$, $\mathcal{O}_n$ is the union of $\mathcal{J}$-classes $J_1, J_2, \ldots, J_{n-1}$, where (for $1 \le r < n$) \[J_r = \{\alpha \in \mathcal{O}_n : |{\rm im \;\alpha}| = r\}.\]
Note that, for $\alpha \in J_{n-1}$, the $\ker \alpha$ has only one non-singleton class with the possibilities $\{1, 2\}, \{2, 3\}, \ldots, \{n - 1, n\}$. Further, every element of $J_{n-1}$ can be uniquely determined by its kernel and image. Hence, for $i \in [n - 1]$ and $k \in [n]$, we use $\zeta_{(i, k)}$ to denote the element of $J_{n-1}$ whose kernel has the non-singleton class $\{i, i + 1\}$ and its image is $[n]\setminus \{k\}$.

\begin{lemma}\label{l.pro-eta}
For $p, r \in [n - 1]$ and $q, s \in [n]$, the product $\zeta_{(p, q)}\zeta_{(r, s)} \in J_{n - 1}$ if and only if $q \in \{r, r + 1 \}$. In this case, $\zeta_{(p, q)}\zeta_{(r, s)} =  \zeta_{(p, s)}$.
\end{lemma}

\begin{proof}
If $q\in \{r, r + 1\}$, then the image of $\zeta_{(p, q)}$ contains precisely one element
from every kernel class of $\zeta_{(r, s)}$ so that $|{\rm im}\;(\zeta_{(p, q)}\zeta_{(r, s)})| = n - 1$. Hence, the result follows.
\end{proof}

\begin{lemma}\label{l.spmset-on}
Every prime subset of $\mathcal{O}_n$ contains at least $n-1$ elements.
\end{lemma}

\begin{proof}
Let $U$ be a prime subset of $\mathcal{O}_n$. By \cite[Lemma 2.1]{a.garba94}, for each $\alpha \in J_r \; (r \leq n-2)$ there exist $\beta, \delta \in J_{r + 1}$ such that $\alpha = \beta \delta$. Hence, the prime subset $U$ contains an element $\zeta_{(p, q)}$ of $J_{n - 1}$, for some $p \in [n-1]$ and $q \in [n]$. In view of Lemma \ref{l.pro-eta}, for each $j \in [n -1]$, let us consider the following $n-1$ decompositions of $\zeta_{(p, q)}$.

\[\zeta_{(p, q)} =
                \left\{\begin{array}{ll}
                \zeta_{(p, j)}\zeta_{(j, q)} & \text {if $q > j $;}  \\
                \zeta_{(p, j + 1)}\zeta_{(j, q)}    & \text {if $q \le j$.}
                \end{array}\right. \]
Note that all the $2(n - 1)$ elements given in the above $n-1$ decompositions are different from each other. Since $U$ is a prime subset, at least one element from each decomposition should be in $U$. Hence, $|U| \ge n - 1$.
\end{proof}

\begin{theorem}
$r_5(\mathcal{O}_n) = {2n - 1\choose n - 1} - n + 1$.
\end{theorem}

\begin{proof}
It is known that $|\mathcal{O}_n| = {2n - 1\choose n - 1} - 1$ (cf. \cite{a.hw71}).  For $q \in [n]$, we observe that the set \[V = \{\zeta_{(i, q)} : i \in [n - 1]\}\] is a prime subset of $\mathcal{O}_n$. Hence, the result follows from Lemma \ref{l.spmset-on} and Corollary \ref{c.tech}.  For $\beta, \delta \in \mathcal{O}_n$, if $\beta \delta \in V$, then $\beta\delta = \zeta_{(k, q)}$ for some $k \in [n-1]$. Then, by Lemma \ref{l.pro-eta}, $\delta = \zeta_{(l, q)}$ for some $l \in [n - 1]$ so that $\delta \in V$.
\end{proof}

\section{Conclusion}

This work introduces the notion of prime subsets in semigroups and provides an alternative technique for finding the large rank of finite semigroups. This technique gives a shorter proof than the proof given by Howie and Ribeiro for the large rank of Brandt semigroups using a graph theoretic approach. Though their approach reveals nice connection between digraphs and Brandt semigroups (cf. \cite{a.hw99,a.hw00}),  the approach has its own limitations with respect to semigroups of transformations. However, in such cases, one may adopt the technique introduced in this work. For instance, this work obtains the large rank of the semigroup of order-preserving singular selfmaps. Further, the present authors found the large ranks of semigroup reducts of affine near-semirings over Brandt semigroups \cite{a.jk-2,a.jk-4}.

\section*{Acknowledgements}

We would like to express our sincere gratitude to the reviewer for his/her encouragement to prepare Section 4 and suggestions which helped us in  improving the presentation of the paper.

\end{document}